\documentclass[12pt]{amsart}
\usepackage{pifont}
\usepackage{mathrsfs}
\usepackage{geometry}
\usepackage{titletoc}
\usepackage{stix2}

\usepackage{amsmath}
\usepackage{amssymb} 
\usepackage{enumitem} 
\usepackage{mathtools} 
\usepackage[table]{xcolor} 
\usepackage[all]{xy} 
\usepackage{tikz} 
\usepackage{tikz-cd}
\usepackage{indentfirst} 
\usepackage{babel} 
\usepackage{setspace} 

\usepackage[colorlinks,linkcolor=red,anchorcolor=green,citecolor=blue]{hyperref} 
\hypersetup{linktocpage = true} 

\usepackage{rotating} 

\usepackage{ytableau} 
\usepackage{longtable} 
\newcolumntype{M}[1]{>{\centering\arraybackslash}m{#1}} 

\usepackage{fancyhdr}

\theoremstyle{plain}
\newtheorem{thm}{Theorem}[section]
\newtheorem{lemma}[thm]{Lemma}
\newtheorem{prop}[thm]{Proposition}
\newtheorem{cor}[thm]{Corollary}
\newtheorem{defn}[thm]{Definition}

\theoremstyle{definition}
\newtheorem{example}[thm]{Example}
\newtheorem{remark}[thm]{Remark}

\newcommand{\btheorem}{\begin{thm}}
	\newcommand{\etheorem}{\end{thm}}
\newcommand{\bproposition}{\begin{prop}}
	\newcommand{\eproposition}{\end{prop}}
\newcommand{\bdefinition}{\begin{defn}}
	\newcommand{\edefinition}{\end{defn}}
\newcommand{\bcorollary}{\begin{cor}}
	\newcommand{\ecorollary}{\end{cor}}
\newcommand{\bproof}{\begin{proof}}
	\newcommand{\eproof}{\end{proof}}
\newcommand{\bremark}{\begin{remark}}
	\newcommand{\eremark}{\end{remark}}
\newcommand{\eexample}{\end{example}}
\newcommand{\bexample}{\begin{example}}

\newcommand{\elemma}{\end{lemma}}
\newcommand{\blemma}{\begin{lemma}}

\newcommand{\p}{\partial}

\renewcommand{\bar}{\overline}

\renewcommand{\phi}{\varphi}

\newcommand{\beq}{\begin{equation}}
\newcommand{\eeq}{\end{equation}}
\newcommand{\ee}{\end{eqnarray*}}
\newcommand{\be}{\begin{eqnarray*}}

\newcommand{\bd}{\begin{enumerate}}
	\newcommand{\ed}{\end{enumerate}}

\renewcommand{\tilde}{\widetilde}

\newcommand{\qtq}[1]{\quad\mbox{#1}\quad}

\newcommand{\Om}{\Omega}

\renewcommand{\S}{{\mathbb S}}

\newcommand{\C}{{\mathbb C}}

\renewcommand{\P}{{\mathbb P}}

\newcommand{\R}{{\mathbb R}}

\setlist[itemize]{leftmargin=*}
\setlist[enumerate]{leftmargin=*}

\numberwithin{equation}{section} 

\makeatletter

\usepackage{fancyhdr}
\renewcommand{\leftmark}{{\scriptsize New curvature characterizations for spherical space forms and  $\C\P^n$}}

\title{New curvature characterizations for spherical space forms and complex projective spaces}

\author{Xiaokui Yang}

\address{Xiaokui Yang, Department of Mathematics and Yau Mathematical Sciences Center, Tsinghua University, Beijing, 100084, China}
\email{xkyang@mail.tsinghua.edu.cn}

\author{Liangdi Zhang}
\address{Liangdi Zhang, Yanqi Lake Beijing Institute of Mathematical Sciences and Applications, Beijing 101408, China}
\email{zld@bimsa.cn}

\begin{document}
	
	\begin{abstract} In this paper, we introduce  a new positivity notion for  curvature  of Riemannian manifolds and obtain characterizations for spherical space forms and the complex projective space $\C\P^n$. 
	\end{abstract}

	\maketitle

	\section{Introduction}

	It is a fundamental topic in Riemannian geometry  to investigate the relationship between the curvature and the topology of Riemannian manifolds.  The classical work of S. Bochner (\cite{Bochner46}) states that a compact Riemannian manifold with positive Ricci curvature must have  vanishing first Betti number. By using the Bochner technique,     S. Gallot and  D. Meyer   proved in  \cite{GM75} (see also \cite{Meyer1971}) that $n$-dimensional compact Riemannian manifolds with positive curvature operators are rational homology spheres, of which the Betti  numbers $b_p$ $(1\leq p\leq n-1)$ vanish.\\

	In  \cite{Tani1974}, A. Tani proved that a compact orientable conformally flat Riemannian manifold with constant scalar curvature and positive Ricci curvature is a spherical space form. Moreover, S. Tachibana \cite{Tachibana74} showed that a compact orientable Riemannian manifold with positive and harmonic curvature operator is a spherical space form.  By using  the Sacks-Uhlenbeck theory of minimal two-spheres,
	M. J. Micallef and M. D. Moore proved in  \cite{MicallefMoore1988} that a compact simply connected Riemannian manifold of dimension at least four, with positive isotropic curvature is homeomorphic to a sphere. Moreover, M. J. Micallef and M. Y. Wang (\cite{MicallefWang1993}) showed that a closed even-dimensional Reimannian manifold with positive isotropic curvature  has $b_2=0$.  Recently, P. Petersen and M. Wink  established in \cite{PetersenWink2021a} more general vanishing theorems and estimation theorems for the $p$-th Betti number of closed $n$-dimensional Riemannian manifolds with a lower bound on the average of the lowest $(n-p)$ eigenvalues of the curvature operator.\\
	
	In 1982, R.  Hamilton  introduced in \cite{Hamilton1982} the Ricci flow method and proved that a $3$-dimensional compact manifold with positive Ricci curvature is a spherical space form.   He  also established in \cite{Hamilton1986} that a $4$-dimensional compact manifold with positive curvature operator is diffeomorphic to a spherical space form. Under the condition of $2$-positive curvature operators, H. Chen (\cite{ChenH1991}) and C. B\"ohm and B. Wilking (\cite{BohmWilking2008}) established the same result in dimension $4$ and  higher dimension respectively. By using \cite{BohmWilking2008},   L. Ni and B. Wu proved in \cite{NiWu2007} that compact Riemannian manifolds with $2$-quasi-positive curvature operators are diffeomorphic to spherical space forms. S. Brendle and R. Schoen confirmed the differentiable sphere theorem in \cite{BrendleSchoen2009}  that a compact Riemannian manifold of dimension $n\geq4$ with strictly $1/4$-pinched sectional curvature is diffeomorphic to a spherical space form.  S. Brendle (\cite{Brendle2010}) proved that a compact Einstein manifold of dimension $n\geq4$ with positive isotropic curvature is a spherical space form. By using the Ricci flow method, compact Riemannian manifolds with positive isotropic curvature have been classified by Hamilton (\cite{Hamilton1997}), Chen-Zhu (\cite{ChenZhu2006}) and Chen-Tang-Zhu (\cite{ChenTangZhu2012}) in dimension $n = 4$. For  the classification when $n\geq 12$, we refer to  Brendle (\cite{Brendle2019}) and  Huang (\cite{Huang2023}).\\


	The first author proved in  \cite{Yang2018}  that a compact K\"ahler manifold of complex dimension $n$ with positive holomorphic sectional curvature is RC-positive and the Hodge numbers $h^{p,0}=0$ for $1\leq p\leq n$ and hence it is projective and rationally connected, which answered affirmatively a conjecture of S. -T. Yau (\cite[Problem~47]{yaup}). Afterwards, L. Ni and F. Zheng  obtained similar results in \cite{NiZheng20a} under the condition of positive orthogonal Ricci curvature. They also proved in \cite{NiZheng20b}  that a compact K\"ahler manifold with positive 2nd-scalar curvature satisfies $h^{2,0}=0$ and it must be projective. More recently, P. Petersen and M. Wink  established a series of vanishing theorems and estimation results in \cite{PetersenWink2021b} for Hodge numbers on compact K\"ahler manifolds with some positive conditions on the eigenvalues of the K\"ahler curvature operator.\\

	Since  Siu-Yau's classical work (\cite{SiuYau1980}) on the solution  to the  Frankel conjecture (see also \cite{Mori1979}), it becomes apparent that the positivity of the tangent bundle of a compact K\"ahler manifold carries important geometric information. In the past decades, many important generalizations have been established. 	For instance, N. Mok  classified compact K\"ahler manifolds with semipositive holomorphic bisectional curvature in \cite{Mok1988}. In  \cite{Gu2009} ,  H. Gu gave a simple and transcendental proof for Mok's theorem by using the K\"ahler-Ricci flow. Moreover, H.  Gu and Z.  Zhang  proved in \cite{GuZhang2010} an extension of the result of \cite{Mok1988} under the condition of nonnegative orthogonal holomorphic bisectional curvature.  H. Feng, K. Liu and X. Wan  gave a new proof in \cite{FengLiuWan2017} that any $n$-dimensional compact K\"ahler manifold with positive orthogonal bisectional curvature must be biholomorphic to $\mathbb{C}\P^n$. For more results on the geometry and topology of various curvature notions,  we refer to \cite{Che07},  \cite{Wil13}, \cite{Li23} and \cite{CGT23} the references therein. \\

	In 2002, P. Guan, J. Viaclovsky and G. Wang \cite{GuanViaclovskyWang2002} proved that the Ricci curvature of an $n$-dimensional Riemannian manifold is positive if the eigenvalues of the Schouten tensor are in  certain cone $\Gamma_k^+$. 
	As an application, they showed that such a compact locally conformally flat Riemannian manifold is conformally equivalent to a spherical space form. Similar rigidity results were obtained in \cite{GuanLinWang2004,GuanWang2004,Gursky1994}. Moreover, P. Guan, C.  Lin and G. Wang \cite{GuanLinWang2005} also proved a cohomology vanishing theorem on compact locally conformally flat Riemannian manifold under certain cone assumption on the Schouten tensor.\\

Let $\Lambda=(a_1,\cdots,a_N)\in\mathbb{R}^N$ and  $\sigma_k$ be the $k$-th elementary symmetric function
$$\sigma_k(\Lambda)=\sum_{i_1<\cdots<i_k}a_{i_1}\cdots a_{i_k}.$$
The $\Gamma_k^+$ cone is defined as
\beq \Gamma_k^+:=\{\Lambda\in\mathbb{R}^N|\ \sigma_j(\Lambda)>0,\ \qtq{for all} j\leq k\}.\eeq 
Let $\bar{\Gamma}_k^+$  be  the closure of $\Gamma_k^+$. Hence 	$\Lambda \in \bar\Gamma^+_k$ if and only if  $\sigma_j(\Lambda)\geq 0$ for all $j\leq k$. It is  easy to see that $\Gamma_N^+\subset \cdots \subset \Gamma_1^+$ and
$\Lambda=(a_1,\cdots,a_N) \in \Gamma^+_N$ if and only if $a_i>0$ for all $i$.\\
	
	In this paper, we study a version of shifted cone on Riemannian manifolds. Let $(M,g)$ be a compact Riemannian manifold of dimension $n\geq 3$. Suppose that $$\lambda_1\leq \cdots \leq \lambda_N$$ are the eigenvalues of the curvature operator $\mathscr R$ on $\Lambda^2TM$ where $N=\frac{n(n-1)}{2}$. If $R$ is the scalar curvature of $(M,g)$, then  $R=2\sum\lambda_i$. For a fixed $\alpha\in \R^+\cup \{0\}$ and $\Lambda=(\lambda_1,\cdots, \lambda_N)$, we define
	 a shifted point in $\R^N$
	\beq \Lambda_{\alpha}:=\Lambda-\alpha \left(\frac{R}{2}, \frac{R}{2},\cdots, \frac{R}{2}\right). \eeq 
	It is obvious that $\Lambda_\alpha=\Lambda$ if $\alpha=0$.
	
	\bdefinition For a  fixed $j\in\{1,2,\cdots,N\}$, we say $g\in\Gamma_j^+(\alpha)$ if $\Lambda_{\alpha}\in \Gamma_j^+$ for all $x\in M$. Similarly,  $g\in\bar \Gamma_j^+(\alpha)$  if $\Lambda_{\alpha}\in \bar \Gamma_j^+$ for all $x\in M$. 
	 \edefinition 	

\noindent	For instance, $g\in\Gamma_1^+(\alpha)$ if and only if \beq \sigma_1\left(\Lambda_\alpha\right)=\sum_{i=1}^N  \left(\lambda_i-\frac{R\alpha}{2}\right)=\frac{(1-\alpha N)R}{2}>0.\eeq 
For  $g\in\Gamma_2^+(\alpha)$,  there is an extra constraint
\beq \sigma_2\left(\Lambda_\alpha\right)=\sum_{1\leq i<j\leq N}  \left(\lambda_i-\frac{R\alpha}{2}\right) \left(\lambda_j-\frac{R\alpha}{2}\right)>0.\eeq 
We will show that for certain  $\alpha$ with geometric meaning,  the positivity of  $\sigma_1\left(\Lambda_\alpha\right)$ and $ \sigma_2\left(\Lambda_\alpha\right)$ can carry important geometric information of $(M,g)$.\\

The classical Weitzenb\"ock formula shows that 	
for any $\omega\in \Om^k(M)$,
\beq  \Delta \omega=D^*D\omega+ dx^i\wedge I_j \Theta \left(\frac{\p}{\p x^i}, \frac{\p}{\p x^j}\right)\omega \label{wei}\eeq
where $\Theta$ is a version of the curvature tensor of the vector bundle  $\wedge^k T^*M$. By using the curvature positivity and formula (\ref{wei}), various rigidity theorems and vanishing theorems are established.  In this paper, we consider the curvature positivity given by the shifted cone $\Gamma_2^+(\alpha)$ with specified parameter $\alpha \in \R^+$. To this end,  we define 
	\beq \alpha_2:=\frac{1}{N}-\frac{1}{N}\sqrt{\frac{2}{(N-1)(N-2)}}. \label{alpha2}\eeq 
The first result of this paper is the following characterization of spherical  space forms.
\btheorem\label{maintheorem1} Let $(M,g)$ be a compact Riemannian manifold of dimension $n\geq 3$. If $g\in {\Gamma}_2^+(\alpha_2)$, then  $M$ is diffeomorphic to a spherical space form.
\etheorem

\noindent This particular number $\alpha_2$ is picked out according to the example $\S^2\times \S^1$ with the canonical product metric $g$. In this case, $g$ lies on the boundary of the cone $\bar \Gamma_2^+(\alpha_2)$  $$\sigma_1(\Lambda_{\alpha_2})>0 \qtq{and} \sigma_2(\Lambda_{\alpha_2})=0.$$
Actually, we obtain a concrete characterization of the  shifted cone $\bar\Gamma_2^+(\alpha_2)$.

\btheorem\label{maintheorem2}
 Let $(M,g)$ be a compact Riemannian manifold of dimension $n\geq 3$. If $g\in \bar \Gamma_2^+(\alpha_2)$, then one of the following holds \bd
\item  $M$ is diffeomorphic to  $\mathbb S^2\times \mathbb S^1$ or $\mathbb R\mathbb P^2\times \S^1$. 
 \item  $(M,g)$ is flat.
\item  $M$ is diffeomorphic to a spherical space form.
\ed
\etheorem

\noindent  More generally,  by considering  $\S^{k}\times \S^1$ with the canonical product metric, we define 
	\beq \alpha_k:=\frac{1}{N}-\frac{1}{N}\sqrt{\frac{k}{(N-1)(N-k)}}, \label{alphak}\eeq 
	and obtain the following curvature positivity and vanishing theorem. 
\btheorem\label{maintheorem3}
 Let $(M,g)$ be a compact Riemannian manifold of dimension $n\geq 3$.  If $g\in {\Gamma}_2^+(\alpha_k)$ for some $1\leq k\leq N-1$,  then  the curvature operator $\mathscr R$ of $(M,g)$ is $k$-positive:
 $$\lambda_1+\cdots+\lambda_k>0.$$
  In particular, \bd
 	\item if $1\leq k\leq\lceil\frac{n}{2}\rceil$, then $b_p(M)=0$ for all $1\leq p\leq n-1$;
 	\item if $\lceil\frac{n}{2}\rceil+1\leq k\leq n-1$, then $$b_1(M)=\cdots=b_{n-k}(M)=0 \qtq{and} b_{k}(M)=\cdots =b_{n-1}(M)=0.$$
 \ed
\etheorem
\noindent 
The vanishing of Betti numbers in Theorem \ref{maintheorem3}   follows from the curvature positivity and the recent work \cite{PetersenWink2021a} of Petersen-Wink. 
	\vskip 1\baselineskip
	
	Let $(M,g)$ be a K\"ahler manifold of complex dimension $n$. It is well-known that by the $J$-invariant property, the curvature operator of the underlying Riemannian manifold $(M,g)$ vanishes on the orthogonal complement of the holonomy algebra $\mathfrak u(n)\subset \mathfrak{so}(2n)$. Hence, the Riemannian curvature operator of a Kähler manifold has a kernel of real
	dimension at least $n(n-1)$. Now we consider the K\"ahler curvature operator  $\mathscr R=\mathscr R|_{\mathfrak{u}(n)}$. Let 
	$$\rho_1\leq \cdots \leq \rho_{n^2}$$
	be the real eigenvalues of the K\"ahler curvature operator. Let $\Lambda=(\rho_1,\cdots, \rho_{n^2})$ and $T=\rho_1+\cdots +\rho_{n^2}$. For
	$k\in\{1,2,\cdots,n^2-1\}$ , we define \beq \beta_k:=\frac{1}{n^2}-\frac{1}{n^2}\sqrt{\frac{k}{(n^2-1)(n^2-k)}}, \eeq and the shifted point \beq \Lambda _{\beta_k}:=\Lambda-\beta_k(T,\cdots,T)\in \R^{n^2}. \eeq 
 We call $g\in \Gamma_2^+(\beta_k)$ if $\Lambda_{\beta_k}\in \Gamma_2^+$ for each point $x\in M$. As analogous to Theorem \ref{maintheorem1}, we obtain a characterization of $\C\P^n$.
 
	\btheorem\label{main4} Let $(M,g)$ be a compact K\"ahler manifold with complex dimension $n\geq 2$. If $g\in \Gamma_2^+(\beta_2)$,  then $M$ is biholomorphic to $\C\P^n$.
	\etheorem
	
	\noindent We also classify compact K\"ahler manifolds with $g\in \bar \Gamma_2^+(\beta_2)$:

	\btheorem\label{maintheorem5}
Let $(M,g)$ be a compact K\"ahler manifold with complex dimension $n\geq 2$. If $g\in \bar \Gamma_2^+(\beta_2)$, then one of the following holds \bd
	\item  $M$ is biholomorphic to  $\C\P^1\times \C\P^1$;
	\item  $(M,g)$ is a flat torus;
	\item  $M$ is biholomorphic to $\C\P^n$.
	\ed
	\etheorem

	\vskip 1\baselineskip 
	
	\noindent \textbf{Ackowledgements}. The authors would like to thank B.-L. Chen and P. Wu for helpful discussions. The first named author is supported by National Key R\&D Program of China 2022YFA1005400 and NSFC grants  (No. 12325103, No. 12171262 and  No. 12141101).

\vskip 2\baselineskip

	\section{Positivity of curvature operators}
	
		In this section, we prove the following result and its K\"ahler analog Theorem \ref{thm00}. As applications, we establish Theorem \ref{maintheorem1}, Theorem \ref{maintheorem2}, Theorem \ref{maintheorem3}, Theorem \ref{main4} and Theorem \ref{maintheorem5}.
	\begin{thm}\label{thm0}
		Let $(M,g)$ be a Riemannian manifold with dimension $n\geq 3$ and
		$$\lambda_1\leq \cdots \leq \lambda_N$$  be the eigenvalues of the curvature operator $\mathscr R$ on the vector space $\Lambda^2T_xM$ where $N={n\choose 2}$ and $x\in M$. Fix $k\in\{1,2,\cdots, N-1\}$.
		\bd \item If $g\in\Gamma_2^+(\alpha_k)$ at point $x\in M$, then $(M,g)$ has $k$-positive curvature operator at $x$: $$\lambda_{1}(x)+\lambda_{2}(x)+\cdots+\lambda_k(x)>0.$$
		\item If $g\in \overline\Gamma_2^+(\alpha_k)$ at point $x\in M$, then one of the followings holds:
		\bd\item $\lambda_{1}(x)+\lambda_{2}(x)+\cdots+\lambda_k(x)>0$.
		\item	$\lambda_{1}(x)=\lambda_{2}(x)=\cdots=\lambda_k(x)=0$ and $\lambda_{k+1}(x)=\lambda_{k+2}(x)=\cdots=\lambda_{N}(x)\geq0$.
		\ed
		
		\ed
	\end{thm}
	
	\vskip 1\baselineskip
	
	\noindent 
	Recall that for $\Lambda=(\lambda_1,\cdots, \lambda_N)\in \R^N$, the shifted point $\Lambda_{\alpha_k}$ is defined as 
	\beq \Lambda_{\alpha_k}=\Lambda-\alpha_k\left(\sum_i \lambda_i, \cdots, \sum_i \lambda_i\right), \eeq 
	where
		\beq \alpha_k=\frac{1}{N}-\frac{1}{N}\sqrt{\frac{k}{(N-1)(N-k)}},\eeq 
	and for $1\leq j\leq N$, the shifted cone $\Gamma_j^+(\alpha_k)$ is defined as 
	\beq \Gamma_j^+(\alpha_k)=\{\Lambda\in \R^N\ |\ \Lambda_{\alpha_k}\in \Gamma_j^+ \}. \eeq 
	
	\bproposition  We have the following properties.
		\bd
			\item If $1\leq k\leq N-1$, then $\Gamma_N^+(\alpha_k)\subset\cdots\subset\Gamma_2^+(\alpha_k)\subset\Gamma_1^+(\alpha_k)$.
		
			\item If $1\leq k_1\leq k_2\leq N-2$, then  $\Gamma_2^+(\alpha_{k_1})\subset \Gamma_2^+(\alpha_{k_2})$  and  $\bar \Gamma_2^+(\alpha_{k_1})\subset \bar \Gamma_2^+(\alpha_{k_2})$.
		\ed
	\eproposition
	\bproof The first part follows from  the definition. For the second part,  let $T=\sum\limits_{i=1}^N\lambda_i$. For any $k\in\{1,2,\cdots,N-2\}$, 
$$\sigma_1(\Lambda_{\alpha_k})=\sum_{i=1}^N(\lambda_i-\alpha_kT)=(1-N\alpha_k)T=T\sqrt{\frac{k}{(N-1)(N-k)}}$$
	which is increasing in $k$. On the other hand, 
\be 	2\sigma_2(\Lambda_{\alpha_k})&=&(\sigma_1(\Lambda_{\alpha_k}))^2-\sum_{i=1}^N(\lambda_i-\alpha_k T)^2\notag\\
		&=&(1-N\alpha_k)^2T^2-\sum_{i=1}^N\lambda_i^2+2T^2\alpha_k-NT^2\alpha_k^2\notag\\
		&=&T^2(1-2(N-1)\alpha_k+N(N-1)\alpha_k^2)-\sum_{i=1}^N\lambda_i^2.
\ee
Since $N\alpha_k<1$ and $\alpha_k$ is decreasing in $k$, we conclude that $\sigma_2(\Lambda_{\alpha_k})$ is increasing in $k$. Therefore, $\Gamma_2^+(\alpha_{k_1})\subset \Gamma_2^+(\alpha_{k_2})$.
	\eproof

	\bexample\label{examplek}  Let $g$ be the canonical product metric on $\S^k\times \S^1$ with $k\geq 2$. Then $\sigma_1(\Lambda_{\alpha_k})>0$ and $\sigma_2(\Lambda_{\alpha_k})=0$. Indeed, the curvature operator of $\S^k\times \S^1$ has eigenvalues
	$$\Lambda=(0,\cdots, 0, 1, \cdots, 1)\in \R^k\times \R^{N-k}.$$
Define 
$$a=-\alpha_k(N-k)=-\frac{N-k}{N}+\frac{1}{N}\sqrt{\frac{k(N-k)}{N-1}},\ \ \ b=a+1=\frac{k}{N}+\frac{1}{N}\sqrt{\frac{k(N-k)}{N-1}}.$$ Hence, we have  $$\Lambda_{\alpha_k}=\Lambda-\alpha_k (N-k,\cdots, N-k)=(a,\cdots,a,b,\cdots,b).$$ A straightforward computation shows
\begin{eqnarray*}
	\sigma_1(\Lambda_{\alpha_k})=ka+(N-k)b=\sqrt{\frac{k(N-k)}{N-1}}>0.
\end{eqnarray*}
Moreover, 
\begin{eqnarray*}
	2\sigma_2(\Lambda_{\alpha_k})&=&\left(\sigma_1(\Lambda_{\alpha_k})\right)^2-\left(ka^2+(N-k)b^2\right)\\
	&=&\frac{k(N-k)}{N-1}-(k(b-1)^2+(N-k)b^2)\\
	&=&\frac{k(N-k)}{N-1}-\left((Nb-2k)b+k\right)=0.
\end{eqnarray*}	
	Hence, $g\in \bar\Gamma_2^+(\Lambda_{\alpha_k})$.
	\eexample

	We have the following result.
	\begin{lemma}\label{lem1} For $1\leq k\leq N-2$,	
		if $\Lambda\in\bar{\Gamma}_2^+(\alpha_k)$, then
		\beq \label{4.000}
		\lambda_{1}+\lambda_{2}+\cdots+\lambda_k\geq0.
		\eeq 
	\noindent Moreover, exactly one of the followings holds.
		\bd \item   $\lambda_{1}+\lambda_{2}+\cdots+\lambda_k>0$.
			\item	$\lambda_{1}=\lambda_{2}=\cdots=\lambda_k=0$ and $\lambda_{k+1}=\lambda_{k+2}=\cdots=\lambda_{N}\geq0$.
		\ed
	\end{lemma}
	
	\begin{proof} For simplicity, we set $T=\sum\limits_{i=1}^N \lambda_i$. Since $\Lambda\in\bar{\Gamma}_2^+(\alpha_k)$, we have
		$$\sigma_1(\Lambda_{\alpha_k})=(1-N\alpha_k)T\geq 0,\ \ \  2\sigma_2(\Lambda_{\alpha_k})=\left(\sigma_1(\Lambda_{\alpha_k})\right)^2-\sum_{i=1}^N\left(\lambda_i-\alpha_kT\right)^2\geq 0.$$
		Since $N\alpha_k<1$, it follows that either  \bd  \item $T=0$ and so $\lambda_1=\cdots =\lambda_N=\alpha_kT=0$; or 
			\item $T>0$.
		\ed
		We  consider the case  $T>0$.  Without loss of generality, we assume  
		\begin{equation}\label{eq8}
		T=\sum_{i=1}^N\lambda_i=N-k.
		\end{equation}
		Otherwise, we replace $\Lambda$ by $s\Lambda$ for some $s>0$.
		Define 
		$$a=-\alpha_k(N-k)=-\frac{N-k}{N}+\frac{1}{N}\sqrt{\frac{k(N-k)}{N-1}},\ \ \ b=a+1=\frac{k}{N}+\frac{1}{N}\sqrt{\frac{k(N-k)}{N-1}},$$
		and $ \Lambda_0=(0,\cdots,0, 1\cdots,1)\in\mathbb{R}^{k}\times\mathbb{R}^{N-k}$. Hence, we have  $$A_{\Lambda_0}:=\Lambda_0-\alpha_k T(1,\cdots, 1)=(a,\cdots,a,b,\cdots,b).$$ We computed in Example \ref{examplek} that 
		$$
			\sigma_1(A_{\Lambda_0})=ka+(N-k)b=\sqrt{\frac{k(N-k)}{N-1}}>0,\ \ \ \sigma_2(A_{\Lambda_0})=0. $$ Hence, $A_{\Lambda_0}\in\bar{\Gamma}_2^+.$\\
		
		We rewrite $\Lambda_{\alpha_k}=(a_1,\cdots,a_N)$ and define $$	A_{\Lambda_t}:=t\Lambda_{\alpha_k}+(1-t)A_{\Lambda_0},\ \  \text{and}\  f(t):=\sigma_2(A_{\Lambda_t}).$$  Since $A_{\Lambda_0}$, $\Lambda_{\alpha_k}\in\bar{\Gamma}_2^+$, the convexity of the cone $\bar{\Gamma}_k^+$ gives that $A_{\Lambda_t}\in\bar{\Gamma}_k^+$ for any $t\in[0,1]$, i.e., $f(t)\geq0$ for any $t\in[0,1]$. 
		It is clear that
		\begin{equation}\label{eq4.2}
		f(0)=\sigma_2(A_{\Lambda_0})=0.
		\end{equation}
		For any $1\leq i\leq N$, Let $(A_{\Lambda_0}|i)$ be the $(N-1)$-vector with the $i$-th component of $A_{\Lambda_0}$ removed. One can see that
		\begin{eqnarray}\label{eq4.03}
		\sigma_1(A_{\Lambda_0}|1)=\cdots=\sigma_1(A_{\Lambda_0}|k)=\sigma_1(A_{\Lambda_0})-a=\sigma_1(A_{\Lambda_0})-b+1
		\end{eqnarray}
		and 
		\begin{eqnarray}\label{eq4.003}
		\sigma_1(A_{\Lambda_0}|k+1)=\cdots=\sigma_1(A_{\Lambda_0}|N)
		&=&\sigma_1(A_{\Lambda_0})-b.
		\end{eqnarray}
		Since 
		\begin{eqnarray*}
			A_{\Lambda_t}=(a+(a_1-a)t,\cdots,a+(a_k-a)t,b+(a_{k+1}-b)t,\cdots,b+(a_N-b)t),
		\end{eqnarray*} one has 
		\begin{eqnarray}\label{eq4.4}
		f'(0)&=&\sum_{i=1}^{k}(a_i-a)\sigma_1(A_{\Lambda_0}|i)+\sum_{j=k+1}^N(a_j-b)\sigma_1(A_{\Lambda_0}|j)\notag\\
		&=&\left(\sum_{i=1}^{k}a_i-ka\right)(\sigma_1(A_{\Lambda_0}|1)-\sigma_1(A_{\Lambda_0}|N))\notag\\
		&=&\left(\sum_{i=1}^{k}a_i\right)-ka=\sum_{i=1}^k \lambda_i,
		\end{eqnarray}
		where in the second identity we use the facts
		$\sum\limits_{i=1}^Na_i=ka+(N-k)b$,
		and
		\begin{equation}\label{eq4.3}
		\sum_{i=k+1}^{N}a_i-(N-k)b=-\sum_{i=1}^{k}a_i+ka.
		\end{equation}
		We know that, for $t\in[0,1]$
		$$f(t)=f(0)+tf'(0)+t^2\frac{f''(0)}{2}=tf'(0)+t^2\frac{f''(0)}{2}\geq 0.$$
		Hence, $f'(0)=\sum\limits_{i=1}^k \lambda_i\geq 0$.
		Furthermore, if  $\lambda_{1}+\lambda_{2}+\cdots+\lambda_{k}=0$,  we shall show $\lambda_{1}=\cdots =\lambda_{k}=0$, and $\lambda_{k+1}=\lambda_{k+2}=\cdots=\lambda_{N}=1$. Indeed, we have 
		\begin{eqnarray*}
			A_{\Lambda_t}&=&(a+(a_1-a)t,\cdots,a+(a_k-a)t,b+(a_{k+1}-b)t,\cdots,b+(a_N-b)t)\notag\\
			&=&(a+t\lambda_1,\cdots,a+t\lambda_k,b+t(\lambda_{k+1}-1),\cdots,b+t(\lambda_{N}-1)).
		\end{eqnarray*}
		By the definition of $f$, we have
		\begin{eqnarray}\label{eq4.6}
		\frac{1}{2}f''(0)&=&\sum_{1\leq i<j\leq k}\lambda_i\lambda_j+\sum_{1\leq i\leq k<j\leq N}\lambda_i(\lambda_j-1)
		+\sum_{k+1\leq i<j\leq N}(\lambda_i-1)(\lambda_j-1)\notag\\
		&=&\frac{1}{2}\left(\sum_{i=1}^k\lambda_i+\sum_{j=k+1}^N(\lambda_j-1)\right)^2-\frac{1}{2}\sum_{i=1}^k\lambda_i^2-\frac{1}{2}\sum_{j=k+1}^N(\lambda_j-1)^2\notag\\
		&=&-\frac{1}{2}\sum_{i=1}^k\lambda_i^2-\frac{1}{2}\sum_{j=k+1}^N(\lambda_j-1)^2,
		\end{eqnarray}
		where in the third identity we use the facts $\sum\limits_{i=1}^k\lambda_i=0$ and $\sum\limits_{j=k+1}^N(\lambda_j-1)=0$. 
		If  $\sum\limits_{i=1}^k\lambda_i=0$, we deduce that $f'(0)=0$. Since $f(t)\geq 0$ for $t\in [0,1]$, one has $f''(0)\geq 0$ and it implies  $\lambda_{1}=\cdots =\lambda_{k}=0$ and $\lambda_{k+1}=\cdots=\lambda_{N}=1$.
	\end{proof}

	\bproposition\label{critical}  Let $(M,g)$ be a compact Riemannian manifold of dimension $n\geq 3$. Suppose $\lambda_1\leq \cdots \leq \lambda_N$ are the eigenvalues of the curvature operator $\mathscr R$ on $\Lambda^2TM$ where $N=\frac{n(n-1)}{2}$. If $\lambda_1=\lambda_2=0$ and $\lambda_3=\cdots =\lambda_N>0$, then $n=3$ and $M$ is diffeomorphic to $\S^2\times \S^1$ or $\R\P^2\times \S^1$. 
	\eproposition
	\bproof Since $(M,g)$ has nonnegative curvature operator, it is well-known that (e.g. \cite[Theorem~7.34]{ChowLuNi2006})  its universal covering manifold $\left(\tilde{M}, \tilde{g}\right)$ is isometric to the product of the following:
	\bd \item 
	 Euclidean space,
	\item  closed symmetric space,
	\item closed Riemannian manifold with positive curvature operator,
	\item closed Kähler manifold with positive curvature operator on real $(1,1)$-forms.
	\ed 
	\vskip 1\baselineskip
	
	Suppose $\left(\tilde{M}, \tilde{g}\right)$ is isometric to a product manifold $(M_1,g_1)\times (M_2, g_2)$ where $\dim M_1=m$. Without loss of generality, we assume $\dim M_1\geq \dim M_2$, i.e.,  $m\geq n-m$. It is easy to see that the number of positive eigenvalues of the curvature operator of $(M_1,g_1)\times (M_2, g_2)$ is at most 
\beq \frac{m(m-1)}{2}+\frac{(n-m)(n-m-1)}{2}=\frac{n(n-1)}{2}-m(n-m). \eeq 
On the other hand,   $\left(\tilde{M}, \tilde{g}\right)$ has $$\frac{n(n-1)}{2}-2$$
 positive eigenvalues. Hence,  
 $$m(n-m)\leq 2.$$
 and we obtain $(m, n)=(2,3)$. When $n=3$, one has $M_1\cong \S^2$ and $M_2\cong \R$. Therefore, $M$ is diffeomorphic to $\S^2\times \S^1$ or $\R\P^2\times \S^1$. \\
 
  Suppose  $\left(\tilde M,\tilde g\right)$ is indecomposable, one can see clearly  that 
   $\left(\tilde M,\tilde g\right)$ must be an irreducible symmetric space of compact type. Hence, it is an Einstein manifold with  constant scalar curvature. Therefore,  $\lambda_3=\cdots =\lambda_N=a$ for some $a\in \R^+$. Let $\{\omega_1,\cdots, w_N\}$ be an orthonormal basis of $\Lambda^2T_x\tilde M$  such that 
  $$\mathscr R(\omega_1)=\mathscr R(\omega_2)=0,\ \ \ \mathscr R(\omega_\alpha)=a\omega_\alpha,\ \ \ \alpha=3,\cdots, N.$$
  For any orthonormal basis $\{e_1,\cdots, e_n\}$ of $T_x\tilde M$, for any $i,j$, we write 
  $$e_i\wedge e_j =\sum_{\alpha=1}^N c_{ij}^\alpha w_\alpha.$$
  Here we have  $c_{ij}^\alpha=-c_{ji}^\alpha $,  $c_{ii}^\alpha=0$ and for any $i\neq j$
  \beq \sum_{\alpha=1}^N\left(c_{ij}^\alpha\right)^2=1. \eeq 
   Moreover, the sectional curvature of it is
\beq \tilde R_{ijji}=\tilde R(e_i,e_j,e_j,e_i)=\tilde g(\mathscr R(e_i\wedge e_j), e_i\wedge e_j)=a\sum_{\alpha=3}^N\left(c_{ij}^\alpha\right)^2 \eeq 
  and the Ricci curvature is 
  \beq \tilde R_{jj}=a \sum_{i=1}^n\sum_{\alpha=3}^N\left(c_{ij}^\alpha\right)^2.  \eeq 
  Note that $(\tilde M, \tilde g)$ is an Einstein manifold with scalar curvature $2(N-2)a$. Hence,  the Einstein constant is $\frac{2(N-2)a}{n}=\frac{\left(n^2-n-4\right)a}{n}$ and so
  $$a \sum_{i=1}^n\sum_{\alpha=3}^N\left(c_{ij}^\alpha\right)^2=\frac{a\left(n^2-n-4\right)}{n}=a\left(n-1-\frac{4}{n}\right).$$
  On the other hand, for fixed $j$, 
 \beq  \sum_{i=1}^n\sum_{\alpha=3}^N\left(c_{ij}^\alpha\right)^2= \sum_{i=1,i\neq j}^n\sum_{\alpha=3}^N\left(c_{ij}^\alpha\right)^2=\sum_{i=1, i\neq j}^n\left(1-\sum_{\alpha=1}^2\left(c_{ij}^\alpha\right)^2\right)=(n-1)-\sum_{i=1}^n\sum_{\alpha=1}^2\left(c_{ij}^\alpha\right)^2.\eeq
  One has
  \beq \sum_{i=1}^n\sum_{\alpha=1}^2\left(c_{ij}^\alpha\right)^2=\frac{4}{n} \qtq{and} 0\leq  \sum_{\alpha=1}^2\left(c_{ij}^\alpha\right)^2\leq \frac{4}{n} \eeq 
  and  the sectional curvature estimate
  \beq \left(1-\frac{4}{n}\right)a\leq  \tilde R_{ijji}=a\sum_{\alpha=3}^N\left(c_{ij}^\alpha\right)^2\leq a.\eeq 
The curvature operator of $(\tilde M, \tilde g)$ has eigenvalues  $\lambda_1=\lambda_2=0$ and $\lambda_3=\cdots =\lambda_N>0$.
  \bd \item 
  For $n\geq 6$, the sectional curvature of $(\tilde M, \tilde g)$  is strictly $1/4$-pinched. By \cite{BrendleSchoen2009}, $\tilde M$ is diffeomeorphic to $\S^n$. It is impossible.
  \item For $n=5$, the sectional curvature of $\tilde M$ is strictly positive. On the other hand,
  homogeneous Ricci positive Einstein $5$-manifolds are classified in \cite[p.~4]{ADM1996}. By using that complete list,  it can not happen.
  
  \item For $n=4$, $(\tilde M,\tilde g)$ is an Einstein manifold with non-negative curvature operator and $3$-positive curvature operator. It is well-known (e.g.
  \cite[Theorem~1.2]{Wu2019}, \cite{BohmWilking2008}, \cite{Brendle2010}, \cite{CaoTran2018})  that such manifolds must be isometric to $\S^4$ or $\C\P^2$ with the canonical metrics up to scalings. It can not happen since the curvature operator of the Fubini-Study metric on $\C\P^2$ has eigenvalues $0,0, 1,1,1,3$. 
  \ed 
 The proof is completed. 
	\eproof 
	
	\vskip 1\baselineskip

		\bcorollary\label{alpha2cone} For $1\leq k\leq N-2$,	
		if $\Lambda\in{\Gamma}_2^+(\alpha_k)$, then
	\begin{equation}
	\lambda_{1}+\lambda_{2}+\cdots+\lambda_k>0.
	\end{equation}
	\ecorollary
	\begin{proof} By Lemma \ref{lem1},  we need to rule out the case $\lambda_{1}=\lambda_{2}=\cdots=\lambda_k=0$ and $\lambda_{k+1}=\lambda_{k+2}=\cdots=\lambda_{N}>0$. Actually, in this case if we assume $T=\sum\limits_{i=1}^N \lambda_i=N-k$, then $\lambda_{k+1}=\lambda_{k+2}=\cdots=\lambda_{N}=1$ and 
		$$\Lambda_{\alpha_k}=(a,\cdots,a,b,\cdots, b)\in \R^k\times \R^{N-k}$$
		as we computed in Example \ref{examplek}. However,  $\sigma_2(A_{\Lambda_0})=0$.
	\end{proof}

	\vskip 1\baselineskip 
	
	\noindent \emph{Proof of Theorem \ref{thm0}.} It follows from Lemma \ref{lem1} and Corollary \ref{alpha2cone}. \qed

	\vskip 1\baselineskip

\noindent	There is a K\"ahler version of Theorem \ref{thm0}.
	\begin{thm}\label{thm00}
		Let $(M,g)$ be K\"ahler manifold with complex dimension $n\geq 2$ and
		$$\rho_1\leq \cdots \leq \rho_{n^2}$$ be the eigenvalues of the K\"ahler curvature operator  $\mathscr R$ at $x\in M$. Fix $k\in\{1,2,\cdots,n^2-1\}$.
		\bd \item If $g\in\Gamma_2^+(\beta_k)$ at $x\in M$, then $$\rho_{1}(x)+\rho_{2}(x)+\cdots+\rho_k(x)>0.$$
			\item If $g\in \bar \Gamma_2^+(\beta_k)$ at $x\in M$, then one of the followings holds:
			\bd  \item $\rho_{1}(x)+\rho_{2}(x)+\cdots+\rho_k(x)>0$.
				\item 	$\rho_{1}(x)=\rho_{2}(x)=\cdots=\rho_k(x)=0$ and $\rho_{k+1}(x)=\rho_{k+2}(x)=\cdots=\rho_{N}(x)\geq0$.
			\ed
		\ed
	\end{thm}
	
\noindent	By using similar arguments as in the proof of Proposition \ref{critical}, we obtain
	\bproposition\label{criticalkahler}  	Let $(M,g)$ be a compact  K\"ahler manifold with complex dimension $n\geq 2$ and
	$\rho_1\leq \cdots \leq \rho_{n^2}$ be the eigenvalues of the K\"ahler curvature operator  $\mathscr R$. If $\rho_1=\rho_2=0$ and $\rho_3=\cdots =\rho_{n^2}>0$, then $n=2$ and $M$ is biholomorphic to $\C\P^1\times \C\P^1$.
\eproposition	

\bproof   Since the underlying Riemannian manifold of $(M,g)$ has nonnegative curvature operator, the holomorphic bisectional curvature of $(M,g)$ is nonnegative. By \cite{Mok1988},
 its universal covering manifold $\left(\tilde{M}, \tilde{g}\right)$ is isometric to the product of the following:
	\bd \item $\C^k$,
	\item   irreducible compact Hermitian symmetric space of rank $\geq 2$,
	\item $\C\P^k$.
	\ed 
	
	Suppose $\left(\tilde{M}, \tilde{g}\right)$ is isometric to a product manifold $(M_1,g_1)\times (M_2, g_2)$ where $\dim_\C M_1=m$. Without loss of generality, we assume $\dim_\C M_1\geq \dim_\C M_2$, i.e.,  $m\geq n-m$. It is easy to see that the number of positive eigenvalues of the K\"ahler curvature operator of $(M_1,g_1)\times (M_2, g_2)$ is at most 
	\beq m^2+(n-m)^2=n^2-2m(n-m). \eeq 
	On the other hand,   $\left(\tilde{M}, \tilde{g}\right)$ has $n^2-2$
	positive eigenvalues. Hence,  
	$$m(n-m)\leq 1.$$
	and we obtain $(m, n)=(1,2)$. It is easy to rule out the case $\C\P^1\times \C$ and we deduce that  $M$ is biholomorphic to $\C\P^1\times \C\P^1$. On the other hand, the product metric on $\C\P^1\times \C\P^1$ is  just a multiply of the product round metric on $\S^2\times \S^2$ which has the desired curvature property.  \\
	
	 Suppose  $\left(\tilde M,\tilde g\right)$ is indecomposable, one can see clearly   that 
	$\left(\tilde M,\tilde g\right)$ must be an irreducible Hermitian symmetric space of compact type. Hence, it is a K\"ahler-Einstein manifold and  $\lambda_3=\cdots =\lambda_{n^2}=a$ for some $a\in \R^+$. By using a similar argument as in the proof of  Proposition \ref{critical}, we deduce that the holomorphic bisectional curvature has uniform bounds
	 \beq \left(1-\frac{2}{n}\right)a\leq \mathrm{HBSC}(\tilde M,\tilde g)\leq a. \label{kahler}\eeq 
	For $n\geq 3$, by \cite{SiuYau1980}, $\tilde M$ is biholomorphic to $\C\P^n$. For $n=2$,  K\"ahler-Einstein Fano surfaces are classified in \cite{Tia90} (see also the exposition in \cite{Tos12}). None of them can match the curvature condition.

Let's describe the proof of (\ref{kahler}) briefly. Let  $\{\omega_1,\cdots \omega_{n^2}\}$ be an orthonormal basis of $T^{1,0}_x\tilde M\wedge T^{0,1}_xM$	
 such that 
$$\mathscr R(\omega_1)=\mathscr R(\omega_2)=0,\ \ \ \mathscr R(\omega_\alpha)=a\omega_\alpha,\ \ \ \alpha=3,\cdots, n^2.$$
For any holomorphic  orthonormal  frame $\{e_1,\cdots, e_n\}$  of $T^{1,0}_x\tilde M$, we have the expansion
$e_i\wedge \bar e_j =\sum\limits_{\alpha=1}^{n^2} c_{i\bar j}^\alpha w_\alpha$ where $\sum\limits_{\alpha=1}^{n^2} |c_{i\bar j}^\alpha|^2=1 $.
 The holomorphic bisectional curvature of $(\tilde M, \tilde g)$ is
\beq \tilde R_{i\bar jj\bar i}=\tilde R(e_i,\bar e_j,e_j,\bar e_i)=\tilde g(\mathscr R(e_i\wedge \bar e_j), e_i\wedge\bar e_j)=a\sum_{\alpha=3}^{n^2}|c_{i\bar j}^\alpha|^2 \eeq 
and the Ricci curvature is 
\beq \tilde R_{j\bar j}=a \sum_{i=1}^n\sum_{\alpha=3}^{n^2}|c_{i\bar j}^\alpha|^2=\frac{(n^2-2)a}{n}=a\left(n-\frac{2}{n}\right)  \eeq 
which implies
\beq \sum_{i=1}^n\sum_{\alpha=1}^2|c_{i\bar j}^\alpha|^2=\frac{2}{n} \qtq{and} 0\leq  \sum_{\alpha=1}^2|c_{i\bar j}^\alpha|^2\leq \frac{2}{n} .\eeq 
The proof is completed.
	\eproof 
	
	\vskip 1\baselineskip
	
		\noindent \emph{Proof of Theorem \ref{maintheorem1}.} If $g\in \Gamma_2^+(\alpha_2)$,  by Theorem \ref{thm0}, we have $$\lambda_1+\lambda_2>0$$ and so the curvature operator $\mathscr R$ of $(M,g)$ is $2$-positive. By the classical result of B\"ohm-Wilking \cite{BohmWilking2008}, $M$ is diffeomorphic to a spherical space form. \qed
		
	\vskip 1\baselineskip

	\noindent \emph{Proof of Theorem \ref{maintheorem2}.} If  $g\in \bar\Gamma_2^+(\alpha_2)$, by  Theorem \ref{thm0}, we achieve the pointwise statement:  at each point $x\in M$, one has either
	\bd \item   $\lambda_{1}(x)+\lambda_{2}(x)>0$;  or
	\item	$\lambda_{1}(x)=\lambda_{2}(x)=0$ and $\lambda_{3}(x)=\lambda_{4}(x)=\cdots=\lambda_{N}(x)\geq0$.
	\ed
	In particular, the curvature operator $\mathscr R$ is $2$-nonnegative.
There are three cases.

\bd \item[(A)] The curvature operator $\mathscr R$ is $2$-quasi-positive, i.e.   $\lambda_{1}(x)+\lambda_{2}(x)\geq 0$ for all $x\in M$ and there exists some $x_0\in M$ such that $\lambda_{1}(x_0)+\lambda_{2}(x_0)>0$.
\item[(B)] $\mathscr R$ is flat, i.e. $\lambda_1(x)=\cdots =\lambda_N(x)=0$ for all $x\in M$.

\item[(C)] For every $x\in M$,  $\lambda_1(x)=\lambda_2(x)=0$ and $\lambda_3(x)=\cdots =\lambda_N(x)>0$.
\ed 	
	For case  (A), by using the method of B\"ohm-Wilking \cite{BohmWilking2008},  Ni-Wu proved in \cite[Corollary~2.3]{NiWu2007} that  along the Ricci flow, the curvature operator $\mathscr R(g(t))$ of $g(t)$ is $2$-positive if $t>0$. In particular, by \cite{BohmWilking2008} again,  $M$ is diffeomorphic to a spherical space form.
	It is clear that $(M,g)$ is flat when case (B) happens.
	
	 For case (C),  we treat it in Proposition \ref{critical} and conclude that $n=3$ and $M$ is diffeomorphic to $\S^2\times \S^1$ or $\R\P^2\times \S^1$.  \qed

	\vskip 1\baselineskip


	\noindent  	\noindent \emph{Proof of Theorem \ref{maintheorem3}.} If $g\in \Gamma_2^+(\alpha_k)$ for some $1\leq k\leq N-1$, by  Theorem \ref{thm0}, the cuvrature operator $\mathscr R$ of $(M,g)$ is $k$-positive. By \cite[Theorem~A]{PetersenWink2021a}, we have 
	\bd \item 
	if $1\leq k\leq\lceil\frac{n}{2}\rceil$, then curvature operator is $\lceil\frac{n}{2}\rceil$-positive and 
	\[b_1(M)=\cdots=b_{\lfloor\frac{n}{2}\rfloor}(M)=b_{\lceil\frac{n}{2}\rceil}(M)=\cdots=b_{n-1}(M)=0,\]
	i.e., $b_p(M)=0$ for all $1\leq p\leq n-1$;
	\item 
	if $\lceil\frac{n}{2}\rceil+1\leq k\leq n-1$,  one concludes that $b_1(M)=\cdots=b_{n-k}(M)=0$ and $b_{k}(M)=\cdots=b_{n-1}(M)=0$.\ed  \qed 
	
	\vskip 1\baselineskip 
	
	\noindent We need an elementary result in the proof of Theorem \ref{main4} and Theorem \ref{maintheorem5}.
	\blemma\label{transition}  Let $(M, g)$ be a Kähler manifold of complex dimension $n$. If the Kähler curvature operator $\mathscr R$ is $2$-positive (resp. $2$-quasi-positive, $2$-nonnegative), then $(M,g)$ has positive (resp. quasi-positive, nonnegative) orthogonal bisectional curvature.
	\elemma 
	
	\bproof  Suppose  that $A=\left(a_{i j}\right)_{m \times m}$ is a real symmetric matrix and the sum of the smallest two eigenvalues of $A$ is positive (resp. nonnegative), then  for any $i \neq j$,  $a_{i i}+a_{j j}>0$ (resp.  $\geq 0$). Indeed, if $\lambda_1 \leq \lambda_2 \leq \cdots \leq \lambda_m$ are the eigenvalues of $A$, then
	$$
\lambda_1+\lambda_2=\inf _{x, y \in \mathbb{R}^m}\{A(x, x)+A(y, y)\ |\ | x|=| y \mid=1,\langle x, y\rangle=0\}.$$
	Let $\left\{\xi_i\right\}$ be the standard basis of $\mathbb{R}^m$.  Then
	$a_{i i}+a_{j j}=A\left(\xi_i, \xi_i\right)+A\left(\xi_j, \xi_j\right) \geq \lambda_1+\lambda_2$.\\
	
	Choose an orthonormal basis $\left\{u_1,  \cdots, u_{2 n}\right\}$ at $T_x M$ such that $J u_i=$ $u_{n+i}$ for $1 \leq i \leq n$. Set $e_i=\frac{1}{\sqrt{2}}\left(u_i-\sqrt{-1} J u_i\right)$, then $\left\{e_1, e_2, \cdots, e_n\right\}$ is a unitary basis of $T_x^{1,0} M$.
	A straightforward computation shows
\be  g\left(\mathscr R\left(e_i \wedge \bar{e}_j+e_j \wedge \bar{e}_i\right), e_i \wedge \bar{e}_j+e_j \wedge \bar{e}_i\right) &=&R_{i \bar{j} j \bar{i}}+ R_{i \bar{j} i \bar{j}}+R_{j \bar{i} {j} \bar{i}}+R_{j \bar{i} i \bar{j}}\\& =& R_{i \bar{j} i \bar{j}}+2 R_{i\bar{i} {j} \bar{j}}+R_{j \bar{i} j \bar{i}}\ee 
	and
	\be 
	 g\left(\mathscr R\left( e_j \wedge \bar{e}_i- e_i \wedge \bar{e}_j\right),  e_j \wedge \bar{e}_i- e_i \wedge \bar{e}_j\right)&=& R_{j \bar{i}i \bar{j}}-R_{j \overline{i }j \bar{i}}-R_{i \bar{j} i \bar{j}} +R_{i \bar{j} {j} \bar{i}}\\
&	= & -R_{j \bar{i} {j} \bar{i}}+2 R_{i \bar{i} j \bar{j}}-R_{i \bar{j} i \bar{j}}.
\ee
Hence, for $i\neq j $,	
\be 4R_{i\bar i j \bar j}&=& g\left(\mathscr R\left(e_i \wedge \bar{e}_j+e_j \wedge \bar{e}_i\right), e_i \wedge \bar{e}_j+e_j \wedge \bar{e}_i\right)\\&&+ g\left(\mathscr R\left(e_j \wedge \bar{e}_i- e_i \wedge \bar{e}_j\right),  e_j \wedge \bar{e}_i- e_i \wedge \bar{e}_j\right). \ee
	Hence, the conclusion follows.
	\eproof

	\noindent \emph{Proof of Theorem \ref{main4}.} If $g\in \Gamma_2^+(\beta_2)$, by Theorem \ref{thm00}, the K\"ahler curvature operator $\mathscr R$  is $2$-positive. By Lemma \ref{transition}, compact  Kähler manifolds with $2$-positive K\"ahler curvature operator have positive
	orthogonal bisectional curvature. By  \cite{GuZhang2010},   they are biholomorphic to $\C\P^n$. \qed

	\vskip 1\baselineskip

	\noindent\emph{Proof of Theorem \ref{maintheorem5}.} By Theorem \ref{thm00},  the K\"ahler curvature operator $\mathscr R$ is $2$-nonnegative. There are three cases.
	
	\bd \item[(A)] The K\"ahler curvature operator $\mathscr R$ is $2$-quasi-positive, i.e.   $\rho_{1}(x)+\rho_{2}(x)\geq 0$ for all $x\in M$ and there exists some $x_0\in M$ such that $\rho_{1}(x_0)+\rho_{2}(x_0)>0$.
	\item[(B)] $\mathscr R$ is flat, i.e. $\rho_1(x)=\cdots =\rho_N(x)=0$ for all $x\in M$.
	
	\item[(C)] For every $x\in M$,  $\rho_1(x)=\rho_2(x)=0$ and $\rho_3(x)=\cdots =\rho_N(x)>0$.
	\ed 	
For case (A), if the K\"ahler curvature operator is $2$-quasi-positive, by Lemma \ref{transition},  the orthogonal  bisectional curvature is quasi-positive.  By \cite{GuZhang2010},  $M$ is biholomorphic to $\C\P^n$. For case (B), it is a flat torus. For case (C), by Proposition \ref{criticalkahler}, $M$ is biholomorphic to $\C\P^1\times \C\P^1$.	\qed

	

	\vskip 2\baselineskip

\end{document}